\documentclass[10pt]{article}
\font\smallit=cmti10  
\usepackage{amssymb,latexsym,amsmath,epsfig} 

\makeatletter

\renewcommand\section{\@startsection {section}{1}{\z@}
{-30pt \@plus -1ex \@minus -.2ex} {2.3ex \@plus.2ex}
{\normalfont\normalsize\bfseries}}

\renewcommand\subsection{\@startsection{subsection}{2}{\z@}
{-3.25ex\@plus -1ex \@minus -.2ex} {1.5ex \@plus .2ex}
{\normalfont\normalsize\bfseries}}

\renewcommand{\@seccntformat}[1]{\csname the#1\endcsname. }
\newtheorem{theorem}{Theorem}[section]
\newtheorem{lemma}[theorem]{Lemma}
\newtheorem{corollary}[theorem]{Corollary}

\newenvironment{proof}
   {\vskip 0.15in \par\noindent{\emph{Proof}.}\hskip 0.5em\ignorespaces}
   {\hfill $\Box$\par\medskip}
\makeatother

\begin{document}

\begin{center}
{\bf THE $n$-COLOR PARTITION FUNCTION AND SOME COUNTING THEOREMS} \vskip 20pt
{\bf Subhajit Bandyopadhyay}\\
{\smallit Department of Mathematical Sciences, Tezpur University, Napaam-784028, Sonitpur, Assam, India}\\
{\tt subha@tezu.ernet.in}\\  \vskip 10pt
{\bf Nayandeep Deka Baruah}\\
{\smallit Department of Mathematical Sciences, Tezpur University, Napaam-784028, Sonitpur, Assam, India}\\
{\tt nayan@tezu.ernet.in}\\
\end{center}
\vskip 30pt

\centerline{\textbf{Abstract}}
Recently, Merca and Schmidt found some decompositions for the partition function $p(n)$ in terms of the classical M\"{o}bius function as well as Euler's totient. In this paper, we define a counting function $T_k^r(m)$ on the set of $n$-color partitions of $m$ for given positive integers $k, r$ and  relate the function with the $n$-color partition function and other well-known arithmetic functions like the M\"obius function, Liouville function, etc. and their divisor sums. Furthermore, we use a counting method of Erd\"os to obtain some  counting theorems for $n$-color partitions that are analogous to those found by Andrews and Deutsch for the partition function.

\pagestyle{myheadings}

\thispagestyle{empty}

\baselineskip=12.875pt

\vskip 30pt

\vspace*{-\baselineskip}

\small

\section{\textbf{Introduction}}\label{secone}  A \emph{partition}
$\lambda=(\lambda_1,\lambda_2,\cdots,\lambda_k)$ of a positive integer
$m$ is a finite sequence of non-increasing positive integers $\lambda_i$, called parts, such that $m=\sum_{i=1}^k\lambda_i$. The partition function $p(m)$ is the number of partitions of $m$. For example, the partitions of $4$ are $(4), ~(3,1),~ (2,2),~
(2,1,1),~ ~ (1,1,1,1)$, and hence, $p(4)=5$.

An $n$-color partition (also called a partition with ``$n$ copies of $n$") of a positive integer $m$ is a partition in which a part of size $n$ can appear in $n$ different colors denoted by subscripts in $n_1,n_2,\ldots, n_n$ and the parts satisfy the order: $$1_1<2_1<2_2<3_1<3_2<3_3<4_1<4_2<4_3<4_4<\cdots.$$
Let  $\textup{PL}(m)$ denote the number of $n$-color partitions of $m$. For example, $\textup{PL}(4)=13$ since there are 13 $n$-color partitions of 4, namely, $(4_1), (4_2), (4_3), (4_4), (3_1,1_1),(3_2,1_1), (3_3,1_1)$, $(2_1,2_1), (2_2,2_1), (2_2,2_2), (2_1,1_1,1_1), (2_2,1_1,1_1)$, and $(1_1,1_1,1_1,1_1)$. The generating  function of $\textup{PL}(m)$  is given by
\begin{align}\label{gen}
\sum_{m=0}^{\infty} \textup{PL}(m)q^m = \prod_{m=1}^{\infty} \dfrac{1}{(1-q^m)^m}.
\end{align}
MacMahon \cite[p. 1421]{macmahon} observed that the right side of \eqref{gen} also generates the number of plane partitions of $m$ (Also see \cite[Corollary 7.20.3]{stanley}), where  a plane partition $\pi$ of $m$ is an array of non-negative integers, $$\begin{array}{cccc}
                                                                        m_{11} & m_{12} & m_{13} & \cdots \\
                                                                        m_{21} & m_{22} & m_{23} & \cdots \\
                                                                        \vdots & \vdots & \vdots &
                                                                      \end{array}
$$
such that $\sum m_{ij}=m$ and the rows and columns are in decreasing order, that is, \linebreak $m_{ij}\ge m_{(i+1)j}$, $m_{ij}\ge m_{i(j+1)}$, for all $i,j\ge1$. For example, the plane partitions of $4$ are

$$\begin{array}{cccccccccccccccccccccccccccccccc}
    4, &  &  & 31,&   &  & 3, &  &  & 2 2&  &  & 2, &  &  & 2  1 1, &  &  & 2  1, &  &  & 2, \\
       &  &  &   &    &  &  1 &  &  &    &  &  & 2  &  &  &         &  &  &  1    &  &  & 1  \\
       &  &  &   &    &  &    &  &  &    &  &  &    &  &  &         &  &  &       &  &  & 1\\
\\  &&  & 111, &   &    & 111, &    &  & 11, &    &  & 11,&    &  &  1. &         &  &  &       &  &  \\
    &&  &      &   &    & \quad1 &    &  & 1 &    &  & 11 &    &  &  1 &         &  &  &       &  &  \\
 &&  &  &   &    & &    &  & 1 &    &  & &    &  &  1 &         &  &  &       &  &  \\
 &&  &  &   &    & &    &  &  &    &  & &    &  &  1 &         &  &  &       &  &  \\
     \end{array}
$$

For further reading on $n$-color partitions and plane partitions we refer to \cite{agarwal-andrews, macmahon, stanley1, stanley2}.

Recently,  Merca and Schmidt \cite{merca-schmidt-mob, merca-schmidt-stan} found some decompositions for the partition function $p(m)$ in terms of the classical elementary functions, namely, the M\"{o}bius function and Euler's totient. In this paper, we find some connections between $n$-color partition function (equivalently, the plane partition function) $\textup{PL(m)}$ and elementary arithmetic functions and their divisor sums.

We define an associated function $T^{r}_{k}(n)$ in two separate scenarios:
\begin{enumerate}
\item
For $r \leq k$, $T^{r}_{k}(n)$ = $\dfrac{1}{k} \times$  (the number of $k$'s in the $n$-color partitions of $n$ with the smallest part at least $r$).
\item
For $r>k$, $T^{r}_{k}(n)$ = the number of the $n$-color partitions of $n-k$ with the smallest part at least $r$ except the possibility of the part $k<r$ being present in only one color.
\end{enumerate}

We consider the following three examples. First we consider $T^{2}_{3}(5)$. We note that the number of $3$'s in the $n$-color partitions of $5$ with the smallest part at least $2$ is $6$, which is evident from the relevant partitions
$3_1+2_1, 3_2+2_2, 3_3+2_1, 3_1+2_2, 3_2+2_1$ and $3_3+2_2$. Therefore, $T^{2}_{3}(5)= \dfrac{1}{3} \times 6 = 2$.

Next, we consider $T^{3}_{2}(5)$. Here $n-k = 5-2 = 3$. The $n$-color partitions of 3 with the smallest part at least 3 except the possibility of the part 2 being present in only one color are $3_1, 3_2$ and $3_3$. Hence, $T^{3}_{2}(5)=3$.

Finally, we consider $T^{3}_{2}(7)$. In this case, $n-k = 7-2 = 5$ and the $n$-color partitions of 5 with the smallest part at least 3 except the possibility of the part 2 being present in only one color are $5_1,  5_2 , 5_3,  5_4, 5_5, 3_1+2, 3_2+2$, $3_3+2$. Thus, $T^{3}_{2}(7) = 8$.

The generating function of $T^{r}_{k}(m)$ is given in the  following lemma.
\begin{lemma}\label{lemma} We have
\begin{align*}
\sum_{m=k}^{\infty} T^{r}_{k}(m)q^m = \frac{q^k}{1-q^k} \prod_{m=r}^{\infty} \frac{1}{(1-q^m)^m}.
\end{align*}
\end{lemma}

We have the following main theorem that connects $\textup{PL}(m)$,  $T_k^r(m)$, and elementary arithmetic functions.

\begin{theorem}\label{main}
Let $A(m)$ be an arithmetic function for $m \geqslant 1$ and $B(m)$ be its divisor sum, that is,
\begin{align*}
B(m)=\sum_{d|m} A(d).
\end{align*}
Also, define the functions $\ell_{r}(m)$ for $m\geqslant0$ and $r\geqslant1$ recursively as
\begin{align}
\ell_{1}(m)&=1,\notag\\\intertext{and}
\label{lr_def}
\ell_{r}(m)&= \sum_{k=0}^{\left \lfloor m/r \right \rfloor} \binom{r+k-1}{k} \ell_{r-1}(m-kr), \text{ for }r\geqslant2.
\end{align}
Additionally, we set
\begin{align*}
\ell_{0}(m)=
\begin{cases}
	1, &\text{if }m=0,\\
	0, &\text{for }m\geqslant1.\\
\end{cases}
\end{align*}
Then for $m \geqslant 1$ and $1 \leqslant r \leqslant m$, we have
\begin{align}
\sum_{k=1}^{m}\textup{PL}(m-k)B(k) = \sum_{k=1}^{m} \sum_{j=0}^{m-k} A(k) T_{k}^{r}(m-j) \ell_{r-1}(j). \label{main1}
\end{align}
\end{theorem}

Recently, Andrews and Deutsch \cite{AD} used a counting technique of Erd\"os to derive certain theorems that involves counting parts of the integer partition. The following theorem is one such result.
\begin{theorem}\label{thm2}
Given $k\ge 1$, Let $S_k(n)$ be the number of appearances of $k$ in the partitions of $n$. Also, in each partition of $n$, we count the number of times a part divisible by $k$ appears uniquely; then sum these numbers over all the partitions of $n$. Let  this number be $S_{|k}(n)$. Then,
$$S_{|k}(n) = S_{2k}(n+k).$$
\end{theorem}

In this paper, we generalize the above theorem to the case of counting the number of times a part congruent to $s(\textup{mod}~k)$ appears uniquely for some $s$
satisfying $0 \leqslant s < k$, then summing these numbers over all the integer partitions of $n$. Furthermore, we apply the same techniques to give counting theorems for $n$-color partitions involving the counting function $T_k^1(n)$, which is a special case of $T_k^r(n)$ defined earlier.

The rest of the paper is organized as follows. In Sections \ref{sec2} and \ref{sec3}, we prove Lemma \ref{lemma} and Theorem \ref{main}, respectively. In Section \ref{sec4}, we  present some corollaries and a detailed work out example. In Section \ref{sec5}, we present an interesting identity involving $\textup{PL}(n)$ and Euler's totient $\phi$ that is analogous to a recent result of Merca and Schmidt \cite{merca-schmidt-stan}. In the final section, we present a generalization of Theorem \ref{thm2} and some counting theorems for $n$-color partitions involving $T_k^1(n)$.

\section{Proof of Lemma \ref{lemma}}\label{sec2}

\begin{proof} If  $G(q)$ denotes the generating function of the number of $n$-color partitions of $m$ with the least part being at least $r$, then
\begin{align*}
G(q)=\prod_{m=r}^{\infty} \frac{1}{(1-q^m)^m}.
\end{align*}
Marking the part $k \geq r$ with a counter $u$, let
\begin{align*}
G(q,u)=\frac{1}{(1-q^r)^r\cdots(1-q^{k-1})^{k-1} (1-uq^k)^k (1-q^{k+1})^{k+1}\cdots}.
\end{align*}
Note that $G(q,1)=G(q)$.
Each term of $G(q,u)$ is of the form $\ell \times u^j \times q^m$ where $j$ is the number of part $k$ present in the $n$-color partitions of $m$ and $\ell$ is the number of such $n$-color partitions where $j$ number of part $k$ are present. If we take derivative with respect to $u$ then the term becomes $\ell \times j \times u^{j-1} \times q^m$ and the terms without $u$ vanishes. Next, taking $u=1$ helps to sum up the $q^m$ terms for each $m$, and we get the required generating function.

Hence, taking symbolic derivative at $u=1$, we obtain the generating function of $k \times T^r_k(m)$ for $r \leq k$ as
\begin{align*}
\left. \frac{dG(u)}{du} \right|_{u=1}
&= \frac{1}{(1-q^r)^r} \cdots \frac{1}{(1-q^{k-1})^{k-1}} \frac{kq^k}{(1-q^k)^{k+1}} \frac{1}{(1-q^{k+1})^{k+1}}\cdots \\
&= \frac{kq^k}{1-q^k} \prod_{m=r}^{\infty} \frac{1}{(1-q^m)^m}.
\end{align*}

In case of $r>k$, we consider $h(m)$ to be the number of the $n$-color partitions of $m$ with the least part being $r$ except the possibility of the part $k<r$ being present in only one color. Then
\begin{align*}
\sum_{m=0}^{\infty} h(m)q^m = \frac{1}{1-q^k} \prod_{m=r}^{\infty} \frac{1}{(1-q^m)^m},
\end{align*}
from which, we have
\begin{align*}
\sum_{m=0}^{\infty} h(m)q^{m+k} = \frac{q^k}{1-q^k} \prod_{m=r}^{\infty} \frac{1}{(1-q^m)^m},
\end{align*}
which can be rewritten, after adjusting the index of the sum on the left side, as
\begin{align*}
\sum_{m=k}^{\infty} h(m-k)q^m = \frac{q^k}{1-q^k} \prod_{m=r}^{\infty} \frac{1}{(1-q^m)^m}.
\end{align*}
Of course, the above gives the required generating function of $T_k^r(m)$ for $r>k$.
\end{proof}

\section{Proof of Theorem \ref{main}}\label{sec3}

Observe that
\begin{align}
\left( \prod_{m=1}^{\infty} \frac{1}{(1-q^m)^m} \right) \left( \sum_{k=1}^{\infty} \frac{A(k)q^k}{1-q^k} \right) &= \left( \sum_{m=0}^{\infty} \textup{PL}(m)q^m \right) \left( \sum_{k=1}^{\infty} B(k)q^k \right) \nonumber\\
&= \sum_{m=1}^{\infty} \left( \sum_{k=1}^{m} \textup{PL}(m-k) B(k) \right) q^m. \label{lhs1}
\end{align}

Again,
\begin{align}
&\left( \prod_{m=1}^{\infty} \frac{1}{(1-q^m)^m} \right) \left( \sum_{k=1}^{\infty} \frac{A(k)q^k}{1-q^k} \right) \nonumber\\
&= \left( \prod_{m=1}^{r-1} \frac{1}{(1-q^m)^m} \right) \left( \prod_{m=r}^{\infty} \frac{1}{(1-q^m)^m} \sum_{k=1}^{\infty} \frac{A(k)q^k}{1-q^k} \right) \nonumber\\
&= \left( \sum_{j=0}^{\infty} \ell_{r-1}(j)q^j \right) \left( \sum_{k=1}^{\infty} A(k) \sum_{m=k}^{\infty} T^{r}_{k}(m)q^m \right) \nonumber\\
&= \left( \sum_{j=0}^{\infty} \ell_{r-1}(j)q^j \right) \left( \sum_{m=1}^{\infty}\sum_{k=1}^m A(k) T^{r}_{k}(m)q^m  \right)\nonumber\\
&= \sum_{m=1}^{\infty} \left( \sum_{j=0}^{m-1}\sum_{k=1}^{m-j} A(k) T^{r}_{k}(m-j)\ell_{r-1}(j) \right)q^m \nonumber\\
&= \sum_{m=1}^{\infty} \left( \sum_{k=1}^m\sum_{j=0}^{m-k} A(k) T^{r}_{k}(m-j)\ell_{r-1}(j) \right)q^m. \label{rhs1}
\end{align}
Comparing \eqref{lhs1} and \eqref{rhs1} we arrived at  the desired result.

Now we work out the definition of $\ell_{r}(m)$. As in the proof, for $r\geqslant1$,
\begin{align*}
\sum_{m=0}^{\infty} \ell_{r}(m)q^m = \prod_{m=1}^{r} \dfrac{1}{(1-q^m)^m},
\end{align*}
from which, for $r\geqslant2$, we see that
\begin{align}\label{lr2}
\sum_{m=0}^{\infty} \ell_{r}(m)q^m &= \left( \sum_{m=0}^{\infty} \ell_{r-1}(m)q^m \right) \left( \frac{1}{(1-q^r)^r} \right) \nonumber\\
&= \left( \sum_{m=0}^{\infty} \ell_{r-1}(m)q^m \right) \left[ \sum_{k=0}^{\infty} \binom{r+k-1}{k} q^{kr} \right] \nonumber\\
&=\sum_{m=0}^{\infty} \left( \sum_{k=0}^{\left \lfloor m/r \right \rfloor} \binom{r+k-1}{k} \ell_{r-1}(m-kr) \right)q^m.
\end{align}

Furthermore,
\begin{align}
\sum_{m=0}^{\infty} \ell_{1}(m)q^m = \frac{1}{1-q} = \sum_{m=0}^{\infty} q^m.\label{lr1}
\end{align}
From \eqref{lr2} and \eqref{lr1}, we arrive at the definition of $\ell_{r}(m)$ for $r\geqslant1$.

It remains to show that our definition for $\ell_{0}(m)$ is consistent with the result. That is, we need to prove that
\begin{align}
\sum_{k=1}^m\textup{PL}(m-k)B(k) = \sum_{k=1}^m A(k) T_{k}^{1}(m). \label{special-case}
\end{align}
We have
\begin{align}
&\left( \prod_{m=1}^{\infty} \frac{1}{(1-q^m)^m} \right) \left( \sum_{k=1}^{\infty} \frac{A(k)q^k}{1-q^k} \right) \nonumber\\
&= \sum_{k=1}^{\infty} A(k) \sum_{m=k}^{\infty} T^{1}_{k}(m)q^m \nonumber\\
&= \sum_{m=1}^{\infty} \left( \sum_{k=1}^{m} A(k) T^{1}_{k}(m) \right)q^n. \label{special-case-rhs1}
\end{align}
Comparing \eqref{lhs1} and \eqref{special-case-rhs1}, we arrive at \eqref{special-case}.

\section{\textbf{Corollaries and an Example}}\label{sec4}
In this section, we substitute $A(m)$ and $B(m)$ in Theorem \ref{main} with some well known pairs of arithmetic functions to arrive at some interesting corollaries.

\begin{corollary}
We have
\begin{align*}
\sum_{k=1}^m\textup{PL}(m-k) =  T_{1}^{1}(m).
\end{align*}
\end{corollary}
\begin{proof}
Taking $A(m)= \left \lfloor{\frac{1}{m}}\right \rfloor $, $B(m) = \sum_{d|m} A(d) = 1$ and  $r=1$ in \eqref{main1} we easily arrive at the corollary.
\end{proof}
\begin{corollary}
For $\mu$ being the M\"obius function, and $m\ge r\geq1$, we have
\begin{align*}
\textup{PL}(m-1) = \sum_{k=1}^m \sum_{j=0}^{m-k} \mu(k) T^{r}_{k}(m-j)\ell_{r-1}(j).
\end{align*}
\end{corollary}
\begin{proof}
Take $A(m) = \mu (m)$. Hence
\begin{align*}
B(m) = \sum_{d|m} \mu (d) = \left \lfloor \dfrac{1}{m} \right \rfloor.
\end{align*}
Putting these in \eqref{main1} we arrive at the required result.
\end{proof}
\begin{corollary}
If $\tau(m)$ is the number of positive divisors of $m$ for $m\geqslant1$, and $m\ge r\geq1$, then
\begin{align*}
\sum_{k=1}^m\textup{PL}(m-k)\tau(k) = \sum_{k=1}^m \sum_{j=0}^{m-k} T_{k}^{r}(m-j) \ell_{r-1}(j).
\end{align*}
\end{corollary}
\begin{proof}
Follows readily by substituting
\begin{align*}
A(m) =1\quad\textup{and}\quad B(m) = \sum_{d|m}A(d)= \sum_{d|m}{1} = \tau(m)
\end{align*} in \eqref{main1}.
\end{proof}
\begin{corollary}
For $\lambda (m)$ being the Liouville function, and $m\ge r\geq1$, we have
\begin{align*}
\sum_{k=1}^{\left \lfloor \sqrt{m} \right \rfloor}\textup{PL}(m-k^2) = \sum_{k=1}^m \sum_{j=0}^{m-k} \lambda (k) T_{k}^{r}(m-j) \ell_{r-1}(j).
\end{align*}
\end{corollary}
\begin{proof}
Let $A(m) = \lambda(m)$. It is well known that
$$B(m) = \sum_{d|m} A(d) = \sum_{d|m} \lambda(d) =\left\{
                                                    \begin{array}{ll}
                                                      1, & \hbox{if $m$ is a square;} \\
                                                      0, & \hbox{otherwise.}
                                                    \end{array}
                                                  \right.
.$$
Substituting the above in \eqref{main1} we readily arrive at the corollary.
\end{proof}
\begin{corollary}
For $\alpha\geqslant1$, let $\sigma_{\alpha}(m) = \sum_{d|m}d^\alpha$. Then, for $m\ge r\geq1$,
\begin{align}
\sum_{k=1}^m\textup{PL}(m-k) \sigma_{\alpha}(k) = \sum_{k=1}^m \sum_{j=0}^{m-k} k^\alpha T_{k}^{r}(m-j) \ell_{r-1}(j).
\end{align}
\end{corollary}
\begin{proof}
Take $A(m) = m^\alpha$ so that $$B(m) = \sum_{d|m} A(d) = \sum_{d|m} d^\alpha = \sigma_{\alpha}(m).$$
We substitute the above in  \eqref{main1} to arrive at the proffered identity.
\end{proof}
\begin{corollary}
For $m\ge r\geq1$,
\begin{align} \label{VonMangoldt}
\sum_{k=1}^m\textup{PL}(m-k) \log k =
\sum_{\substack{1 < {k} \leqslant m, \\ k=p^c, p\text{ prime}, c \geqslant 1}} \sum_{j=0}^{m-k} T_{k}^{r}(m-j) \ell_{r-1}(j)\log p.
\end{align}
\end{corollary}
\begin{proof}
Take $A(m)= \Lambda(m)$, the Von Mangoldt function. Then
$$B(m) = \sum_{d\mid m} \Lambda(m) = \log m.$$
The result now follows by substituting these in  \eqref{main1}.
\end{proof}

\noindent \textbf{Example}.
We work out the case $m=11$ and $r=3$ in \eqref{VonMangoldt}.

First of all, we generate the required values for $\ell_2(j)$. Using the definition \eqref{lr_def},  we have
\begin{align*}
\ell_2(m) &= \sum_{k=0}^{\lfloor m/2 \rfloor} \binom{2+k-1}{k} \ell_{1}(m-2k) = \sum_{k=0}^{\lfloor m/2 \rfloor} (k+1) \\
       &= 1 + 2 + ... + (\lfloor m/2 \rfloor + 1).
\end{align*}
Using the above, we have the values as shown in the following table.

\vspace{.5cm}\begin{tabular}{| l | c | c | c | c | c | c | c | c | c | c |}
  \hline			
       $j$ & 0 & 1 & 2 & 3 & 4 & 5 & 6  & 7  & 8  & 9 \\ \hline
  $\ell_2(j)$ & 1 & 1 & 3 & 3 & 6 & 6 & 10 & 10 & 15 & 15 \\
  \hline
\end{tabular}

\vspace{.5cm} Setting $m=11$, the left hand side of \eqref{VonMangoldt} becomes
\begin{align*}
\sum_{k=1}^{11}\textup{PL}(11-k)\log k.
\end{align*}

\vspace{.5cm} \noindent The corresponding coefficients of the $\log k$ terms are given in the following table.

\vspace{.5cm}\centerline{Table 3.1}

\vspace{.5cm}

\begin{tabular}{| l | c | c | c | c | c | c | c | c | c | c |}
  \hline			
       $\log k$ & $\log 2$ & $\log 3$ & $\log 5$ & $\log 7$ & $\log 11$ \\ \hline
    corresponding coefficients      & 497       & 190       & 49        & 13        & 1          \\
  \hline
\end{tabular}

\vspace{.5cm} Setting $m=11$ and $r=3$, the right hand side of \eqref{VonMangoldt} becomes
\begin{align*}
\sum_{\substack{1 < {k} \leqslant 11, \\ k=p^c, p\text{ prime}, c \geqslant 1}} \sum_{j=0}^{11-k}  T_{k}^{3}(11-j) \ell_{2}(j) \log p.
\end{align*}

\noindent The coefficients of the $\log p$ terms are given in the following table, which matches with the values in Table 3.1 calculated for the left hand side of \eqref{VonMangoldt} for $m=11$.

\vspace{.5cm}\centerline{Table 3.2}

\vspace{.5cm} \begin{tabular}{| l | c |}
\hline
$\log p$ & corresponding coefficients \\
\hline
$\log 2$      & $\sum_{j=0}^{9}T_{2}^{3}(11-j) \ell_{2}(j) + \sum_{j=0}^{7}T_{4}^{3}(11-j) \ell_{2}(j)$ \\ &$+ \sum_{j=0}^{3}T_{8}^{3}(11-j) \ell_{2}(j)$ = 497 \\
\hline
$\log 3$      & $\sum_{j=0}^{8}T_{3}^{3}(11-j) \ell_{2}(j) + \sum_{j=0}^{2}T_{9}^{3}(11-j) \ell_{2}(j)$ = 190\\
\hline
$\log 5$      & $\sum_{j=0}^{6}T_{5}^{3}(11-j) \ell_{2}(j)$ = 49\\
\hline
$\log 7$      & $\sum_{j=0}^{4}T_{7}^{3}(11-j) \ell_{2}(j)$ = 13\\
\hline
$\log 11$     & 1\\
\hline
\end{tabular}

\vspace{.5cm}
As a demonstration, we now explicitly calculate the coefficient of $\log 3$, that is 190,  from the combinatorial procedure.

\noindent To this end, for the partitions of 11 with the smallest part at least 3, the following table helps to calculate $T_{3}^{3}(11)$.

\vspace{.5cm}
\begin{tabular}{| c | c | c |}
\hline
Integer partition  &  Number of corresponding     &  Total number of \\
  &  $n$-color partitions    &  parts 3 present \\
\hline
(8,3)                &  24                                              &  24  \\
(5,3,3)              &  30                                              &  60  \\
(4,4,3)             &  30                                              &  30  \\
\hline
\end{tabular}

\vspace{.5cm}
\noindent Thus, $T_{3}^{3}(11) = \dfrac{1}{3} (24+60+30) = 38$.

Next, for the partitions of 10 with the smallest part at least 3, we have

\vspace{.5cm}
\begin{tabular}{| c | c | c |}
\hline
Integer partition  &  Number of corresponding     &  Total number of \\
  &  $n$-color partitions    &  parts 3 present \\
\hline
(7,3)                &  21                                              &  21  \\
(4,3,3)              &  24                                              &  48  \\
\hline
\end{tabular}

\vspace{.5cm}
\noindent Therefore, $T_{3}^{3}(10) = \dfrac{1}{3} (21+48) = 23$. \\

For the partitions of 9 with the smallest part at least 3, the following table helps to calculate $T_{3}^{3}(9)$.

\vspace{.5cm}\begin{tabular}{| c | c | c |}
\hline
Integer partition  &  Number of corresponding     &  Total number of \\
  &  $n$-color partitions    &  parts 3 present \\
\hline
(6,3)                &  18                                              &  18  \\
(3,3,3)              &  10                                              &  30  \\
\hline
\end{tabular}

\vspace{.5cm}
\noindent Hence, $T_{3}^{3}(10) = \dfrac{1}{3} (18+30) = 16$. \\

For the partitions of 8 with the smallest part at least 3, we have the following table that helps to calculate $T_{3}^{3}(8)$.

\vspace{.5cm}
\begin{tabular}{| c | c | c |}
\hline
Integer partition  &  Number of corresponding     &  Total number of \\
  &  $n$-color partitions    &  parts 3 present \\ \hline
(5,3)                &  15                                              &  15  \\
\hline
\end{tabular}

\vspace{.5cm}
\noindent Thus, $T_{3}^{3}(10) = \dfrac{1}{3} \times 15 = 5$.

In a similar way, we calculate $T_{3}^{3}(7)=4, T_{3}^{3}(6)=4, T_{3}^{3}(3)=1$ and $T_{9}^{3}(9)=1$.

Now, using the table of $\ell_2(j)$ and the above values, we arrive at
\begin{align*}
\sum_{j=0}^{8}T_{3}^{3}(11-j) \ell_{2}(j) + \sum_{j=0}^{2}T_{9}^{3}(11-j) \ell_{2}(j) &= 187 + 3 =190,
\end{align*}
which coincides with the coefficient 190 of $\log 3$ in Table 3.2.

\section{A Special Identity involving Euler's totient $\phi$}\label{sec5}
We recall from Hardy and Write \cite[Theorem 309]{hardy} that,
\begin{align}
\sum_{m=0}^{\infty} \frac{\phi(m)q^m}{1-q^m} = \frac{q}{(1-q)^2}. \label{phiclosedform}
\end{align}
Due to the existence of such a closed form, we can pursue a different approach for the case of $\phi$ function, as done in the paper by Merca and Schmidt \cite{merca-schmidt-stan}.
\begin{theorem}\label{phi-thm}
For $m\geqslant0$,
\begin{align} \label{particular}
\textup{PL}(m+2) - \textup{PL}(m) = \dfrac{1}{2} \sum_{k=3}^{m+5} \phi(k) T^{3}_{k}(m+5).
\end{align}
\end{theorem}
\begin{proof}
Notice that
\begin{align}
&(1-q)(1-q^2)^2 \prod_{m=1}^{\infty} \dfrac{1}{(1-q^m)^{m}} \sum_{k=1}^{\infty} \frac{\phi(k)q^k}{1-q^k} \nonumber\\
&= (q+q^2-3q^3-q^4+2q^5) \prod_{m=1}^{\infty} \dfrac{1}{(1-q^m)^{m}} + \sum_{m=3}^{\infty} \sum_{k=3}^m \phi(k) T^{3}_{k}(m)q^m. \label{particular1}
\end{align}

Again, using the closed form \eqref{phiclosedform}, we have
\begin{align}
&(1-q)(1-q^2)^2 \prod_{m=1}^{\infty} \dfrac{1}{(1-q^m)^{m}}  \sum_{k=1}^{\infty} \dfrac{\phi(k)q^k}{1-q^k}\notag\\
&= (q+q^2-q^3-q^4) \prod_{m=1}^{\infty} \dfrac{1}{(1-q^m)^{m}}  \label{particular2}
\end{align}
From  \eqref{particular1} and \eqref{particular2}, we find that
\begin{align*}
\sum_{m=3}^{\infty} \left( \sum_{k=3}^m \phi(k) T^{3}_{k}(m) \right) q^m &= 2(q^3-q^5) \left( \sum_{m=0}^{\infty}\textup{PL}(m)q^m \right) \\
&= 2 \left( \sum_{m=3}^{\infty} \textup{PL}(m-3) - \sum_{m=5}^{\infty} \textup{PL}(m-5) \right) q^m.
\end{align*}
Equating the coefficients of $q^{m+5}$, for $m\ge 0$, from both sides of the above, we readily arrive at \eqref{particular} to finish the proof.
\end{proof}

\noindent \textbf{Example}. We verify Theorem \ref{phi-thm} for the case $m=6$.

The left side of \eqref{particular} is $\textup{PL}(8) - \textup{PL}(6) = 160 - 48 = 112$.

On the other hand, the right side of \eqref{particular} can be worked out as
\begin{align*}
&\dfrac{1}{2} \sum_{k=3}^{11} \phi(k) T^{3}_{k}(11) \\
&= \frac{1}{2} \big( \phi(3) T^{3}_{3}(11) + \phi(4) T^{3}_{4}(11) + \phi(5) T^{3}_{5}(11) + \phi(6) T^{3}_{6}(11) + \phi(7) T^{3}_{7}(11) \\
&\quad + \phi(8) T^{3}_{8}(11) + \phi(11) T^{3}_{11}(11) \big) \\
&= \dfrac{1}{2} \left( 2 \times 38 + 2 \times 22 + 4 \times 12 + 2 \times 5 + 6 \times 4 + 4 \times 3 + 10 \times 1 \right) \\
&= 112.
\end{align*}
Thus the result is verified for $n=6$.

\section{A generalization of Theorem \ref{thm2} and some counting theorems for $n$-color partitions}

We state a generalization of Theorem \ref{thm2} as follows.
\begin{theorem}
Given $k\geqslant1$, in each partitions of $n$ we count the number of times a part congruent to $s(mod$ $k)$ appears uniquely for some $s$ satisfying $0 \leqslant s < k$, then sum these numbers over all the partitions of $n$. Let us call this $S_{s(k)}(n)$. Then
$$ S_{s(k)}(n) = S_{2k}(n+k-s)+S_{2k}(n-s)-S_{2k}(n-2s).$$
\end{theorem}
\begin{proof}
Approaching as in \cite{AD}, assuming $n\geqslant1$, $k\geqslant1$, the generating function of $S_{s(k)}(n)$ is given by
\begin{align*}
\sum_{n\geqslant1}S_{2k}(n)q^n = \sum_{j=1}^{\infty} \frac{q^{kj+s}}{\prod_{n \neq kj+s}(1-q^n)},
\end{align*}
from which, we have
\begin{align*}
\sum_{n\geqslant1}S_{2k}(n)q^n &= \frac{1}{\prod_{n}(1-q^n)} \sum_{j=1}^{\infty}q^{kj+s}(1-q^{kj+s}) \\
&= \frac{1}{\prod_{n}(1-q^n)} \left( q^s \frac{q^k}{1-q^k} - q^{2s} \frac{q^{2k}}{1-q^{2k}} \right) \\
&= \frac{q^{2k}}{1-q^{2k}} \frac{1}{\prod_{n}(1-q^n)} \left(q^{-(k-s)} + q^{s} - q^{2s} \right) \\
&= \left(\sum_{n\geqslant1}S_{2k}(n)q^n\right)\left(q^{-(k-s)} + q^{s} - q^{2s} \right).
\end{align*}
Comparing the coefficients of $q^n$ from both sides, we arrive at the desired result.
\end{proof}

The following theorem presents the case for $n$-color partitions in the spirit of Theorem \ref{thm2}. This shows how functions of the form $T_{k}^{r}(n)$ can be useful in such counting theorems.
\begin{theorem} \label{special-case-T}
In each $n$-color partitions of $n$, we count the number of times a part divisible by k appears uniquely, then sum these numbers over all the $n$-color partitions of $n$. Let us multiply this sum by $\frac{1}{k}$ and call it $T_{|k}(n)$. Then
$$T_{|k}(n) = \sum_{j \geqslant 0} \Big(T_{2k}^1(n-(2j-1)k) + T_{2k}^1(n-2jk) + T_{2k}^1(n-(2j+1)k)\Big).$$
\end{theorem}
\begin{proof}
The generating function of $T_{|k}(n)$ is given by
$$k\sum_{n}T_{|k}(n)q^n = \sum_{j=1}^{\infty} \frac{kjq^{kj}}{(1-q^{kj})^{kj-1} \prod_{n \neq kj}(1-q^n)^n}.$$
Hence,
\begin{align}\label{T2K}
\sum_{n}T_{|k}(n)q^n &= \frac{\sum_{j=1}^{\infty} jq^{kj} (1-q^{kj})}{\prod_{n}(1-q^n)^n} \nonumber\\
&= \frac{1}{\prod_{n}(1-q^n)^n} \Big( \sum_{j=1}^{\infty} jq^{kj} - \sum_{j=1}^{\infty} jq^{2kj} \Big) \nonumber\\
&= \frac{1}{\prod_{n}(1-q^n)^n} \left( \frac{q^k}{(1-q^k)^2} - \frac{q^{2k}}{(1-q^{2k})^2} \right) \nonumber\\
&= \frac{1}{\prod_{n}(1-q^n)^n} \frac{q^{2k}}{(1-q^{2k})^2} \Big( q^{-k} + 1 + q^{k} \Big) \nonumber\\
&= \Big( \sum_{n}T_{2k}^1(n)q^n \Big) \Big( \sum_{j=1}^{\infty}( q^{(2j-1)k} + q^{2jk} + q^{(2j+1)k} ) \Big).
\end{align}
Comparing the coefficients of $q^n$ from both sides of \eqref{T2K}, we obtain the desired result.
\end{proof}

In fact, we can also generalize this theorem to any part congruent to $s(\textup{mod}~k)$ as follows.
\begin{theorem} \label{general-case}
In each $n$-color partitions of $n$ we count the number of times a part congruent to $s(\textup{mod}~k)$  appears uniquely for some $s$ satisfying $0 \leqslant s < k$, then sum these numbers over all the $n$-color partitions of $n$. Let us call this $T_{s(k)}(n)$. Then,
\begin{align*}
T_{s(k)}(n) = &(k+s)(T_{2k}^1(n+k-s) - T_{2k}^1(n-2s)) + sT_{2k}^1(n-s) \\
&+ 2k \sum_{l \geqslant 1}T_{2k}^1(n+k-s-kl) - k \sum_{l \geqslant 1}T_{2k}^1(n-2s-2kl).
\end{align*}
\end{theorem}
\begin{proof}
The generating function of $T_{s(k)}(n)$ is given by
\begin{align*}
\sum_{n\geqslant0} T_{s(k)}(n) q^n&
= \sum_{j\geqslant1} \frac{(kj+s)q^{kj+s}}{(1-q^{kj+s})^{kj+s-1} \prod_{n \neq kj+s} (1-q^n)^n}.
\end{align*}
Therefore,
\begin{align*}
&\sum_{n\geqslant0} T_{s(k)}(n) q^n \\
&= \sum_{j\geqslant1} \frac{(kj+s)q^{kj+s}}{(1-q^{kj+s})^{kj+s-1} \prod_{n \neq kj+s} (1-q^n)^n} \\
&= \sum_{j\geqslant1} \frac{(kj+s)q^{kj+s}(1-q^{kj+s})}{\prod_{n \geqslant 1} (1-q^n)^n} \\
&= \frac{1}{\prod_{n \geqslant 1} (1-q^n)^n} \left(kq^s \sum_{j\geqslant1}jq^{kj} + sq^s\sum_{j\geqslant1}q^{kj} - kq^{2s}\sum_{j\geqslant1}jq^{2kj} - sq^{2s}\sum_{j\geqslant1}q^{2kj} \right) \\
&= \frac{1}{\prod_{n \geqslant 1} (1-q^n)^n} \left(kq^s \frac{q^{k}}{(1-q^{k})^2} + sq^s \frac{q^{k}}{1-q^{k}} - kq^{2s} \frac{q^{2k}}{(1-q^{2k})^2} - sq^{2s} \frac{q^{2k}}{1-q^{2k}} \right) \\
&= \frac{q^{2k}}{1-q^{2k}} \frac{1}{\prod_{n \geqslant 1} (1-q^n)^n} \left( kq^{-(k-s)} \frac{1+q^k}{1-q^k} + sq^{-(k-s)}(1+q^k) - kq^{2s} \frac{1}{1-q^{2k}} - sq^{2s} \right)\\
&= \left( \sum_{n\geqslant0}T_{2k}^1(n)q^n \right) \left( kq^{-(k-s)} \left(1+2\sum_{l\geqslant1}q^{kl}\right) + sq^{-(k-s)}(1+q^k) - kq^{2s}\left(l+\sum_{l\geqslant1}q^{2kl} \right) \right) \\
&= \left( \sum_{n\geqslant0}T_{2k}^1(n)q^n \right) \left( kq^{-(k-s)} + sq^{-(k-s) + sq^s - kq^{2s} - sq^{2s} + 2kq^{-(k-s)}}\sum_{l\geqslant1}q^{kl} - kq^{2s}\sum_{l\geqslant1}q^{2kl}  \right)\\
&=  (k+s)\sum_{n}T_{2k}^1(n)q^{n-k+s} + s\sum_{n}T_{2k}^1(n)q^{n+s} 
- (k+s)\sum_{n}T_{2k}^1(n)q^{n+2s} \\
&\quad + 2kq^{-(k-s)} \sum_{n} \left( \sum_{l\geqslant1}T_{2k}^1(n-kl) \right)q^n - kq^{2s} \sum_{n} \left( \sum_{l\geqslant1}T_{2k}^1(n-2kl) \right)q^n
\end{align*}
\begin{align*}
&= (k+s)\sum_{n}T_{2k}^1(n+k-s)q^{n} + s\sum_{n}T_{2k}^1(n-s)q^{n} - (k+s)\sum_{n}T_{2k}^1(n-2s)q^{n} \\
&\quad+ 2k \sum_{n} \left( \sum_{l\geqslant1}T_{2k}^1(n+k-s-kl) \right)q^n - k \sum_{n} \left( \sum_{l\geqslant1}T_{2k}^1(n-2s-2kl) \right)q^n. \\
\end{align*}
Comparing the coefficient of $q^n$ from both sides, we arrive at the desired result.
\end{proof}

It is to be noted that Theorem \ref{special-case-T} can be obtained from Theorem \ref{general-case} by putting $s=0$, taking $T_{|k}(n) = \frac{1}{k}T_{0(k)}(n)$, and then rearranging the terms.

\section*{\bf Acknowledgement}
 The first  author was partially supported by an ISPIRE Fellowship for Doctoral Research,
 DST, Government of India. The second author was partially supported by Grant no. MTR/2018/ 000157/MS of Science \& Engineering Research Board (SERB), DST, Government of India, under MATRICS Scheme. The authors acknowledge the respective funding agencies.

\end{document}